\documentclass[12pt]{amsart}
\usepackage{graphicx}
\usepackage[dvipsnames]{color}
\definecolor{gris}{gray}{0.7}
\usepackage{color}
\definecolor{nar}{rgb}{.8,0,1}
\definecolor{otro}{rgb}{0,1,1}
\usepackage{amssymb}
\usepackage{amsmath}
\usepackage{multicol}
\usepackage{epic}
\usepackage{fancybox}
\usepackage{url}
\usepackage{bbm}
\usepackage{dsfont}
\usepackage{eurosym}
\usepackage{slashed}
 \usepackage[T1]{fontenc}
\usepackage[latin1]{inputenc}
\usepackage[frenchb,english]{babel}
\usepackage[dvipsnames]{color}
\definecolor{gris}{gray}{0.7}

\newcommand{\D}{\displaystyle}
\newcommand{\R}{\mathbb{R}}

\newcommand{\angnabla}{\slashed{\nabla}}

\newtheorem{theorem}{Theorem}
\newtheorem{remark}{Remark}
\newtheorem{lemma}{Lemma}
\newtheorem{proposition}{Proposition}

\begin{document}

\title[Scattering for NLS outside a 2d star--shaped obstacle]
{Scattering for solutions of NLS in the exterior
of a 2d star--shaped obstacle}

  \author{Fabrice Planchon}
  \address{Laboratoire J. A. Dieudonn\'e, UMR 7351\\
    Universit\'e de Nice Sophia-Antipolis\\
    Parc Valrose\\
    06108 Nice Cedex 02\\
    France and Institut universitaire de France} \email{fabrice.planchon@unice.fr} \thanks{The first
    author was partially supported by A.N.R. grant SWAP}

\author{Luis Vega}
\address{Universidad del Pa\'is Vasco, Departamento deMatem\'aticas,\\
 Apartado 644, 48080, Bilbao, Spain}
\email{luis.vega@ehu.es}\thanks{The 
   second author was partially supported by UPV/EHU and by MEC grant MTM 2007--03029}

 \date{}

\begin{abstract}
   We prove that solutions to  non-linear Schr\"odinger equations in
   two dimensions and in the exterior of  a bounded and smooth star-shaped
   obstacle scatter in the energy space. The non-linear potential is
   defocusing and grows at least as the quintic power.\end{abstract}

\maketitle
  
   \section{Introduction and main results.}
  \label{S1}
  
  In this paper we are interested in the initial value problem
  \begin{equation}
\label{eq1}
\left\{\begin{array}{rcl}
i\partial_t u+\Delta u&=&\epsilon |u|^{p-1}u\qquad u\in\R^2\,\setminus\,\Sigma=\Omega\quad,\quad \epsilon\in\{0,1\},\quad p\geq5,\\
\left. u\right|_{\partial \Sigma}&=&0\\
u(x,0)&=&u_0(x)\in H_0^1(\Omega)
\end{array}\right.
\end{equation}
  and more particularly in the proof  
  of scattering properties of the solutions. With respect to $\Sigma$
  we will assume that it is a star--shaped domain contained in a
  compact set $K$, and that its boundary
  $\partial\Sigma=\partial\Omega$ is a smooth curve in $\R^2$.
  
  As it is well--known dimensions one and two are the most delicate
  ones regarding scattering questions even if no obstacle is
  considered and one works in the full space $\R^d$. The main
  obstruction comes from the sign of the bilaplacian term that appears
  in the classical Morawetz-Lin-Strauss inequality \cite{LS}.  This
  sign turns out to be the wrong one for $d=1, 2$. In dimensions three
  and higher and in $\R^d$ the question was settled in the energy
  space by J. Ginibre and G. Velo in \cite{GV1} .  The obstruction in
  the dimension was finally removed, still in full space, by
  K. Nakanishi, \cite{N}, but his technique does not seem directly applicable in the
  exterior of a 2D domain; that is our interest here. One should
  remark that the 1D case (with both Dirichlet or Neuman boundary
  conditions) follows directly from the $\R$ case, as symmetry with
  respect to $x=0$ is preserved by the nonlinear flow. Hence for
  exterior domains the 2D case may be seen as the most difficult one.

  Another fundamental contribution regarding global existence and
  scattering was the introduction by Colliander, Keel, Staffilani,
  Takaoka, and Tao in \cite{CKSTT1} and \cite{CKSTT2} of interactive
  Morawetz inequalities. However a similar obstruction in  low
  dimensions appears as well. Again the difficulty comes from the sign of
  the bilaplacian term that appears in the  use of the bilinear
  multiplier. It turns out that this obstruction can be easily 
  bypassed (see \cite{CGT} and \cite{PV}) and as a consequence the
  scattering of NLS equations for $L^2$ subcritical non--linearities
  is by now fairly well understood (see for example \cite {GV2} for a good
  survey, and also \cite{CHVZ} for earlier results in 1D using a
  quadrilinear interactive Morawetz inequality).
  
  In \cite{PV} the bilinear multiplier technique is also used in the
  setting of exterior domains with Dirichlet boundary conditions. In
  order to explain our results we have to introduce some notation.
  
  Let 
  $$n\ge1\quad,\quad p>1\quad,\quad \epsilon\in\{-1,0,1\}\quad,\quad
  \Omega_1,\Omega_2\subset \R^n$$
  with smooth boundaries $\partial \Omega_1,\partial\Omega_2$, and $u_1,u_2$ solutions of
   \begin{equation}
\label{eq2}
i\partial_t u_j+\Delta u_j=\epsilon \left|u_j\right|^{p-1}u_j\qquad\text{with }\, \left. u_j\right|_{\partial \Omega_j}=0\,\quad j=1,2.
\end{equation}

Denote for $j=1,2$
\begin{equation}
\label{eq3}
M(u_j)=\D\int_{\Omega_j}\left|u_j\right|^2dx\,;
\qquad
E(u_j)=\D\frac 12\int_{\Omega_j}\left|\nabla u\right|^2+\frac \epsilon{p+1}\int_{\Omega_j}\left|u_j\right|^{p+1}dx,
\end{equation}
the mass and energy that are conserved quantities. As in \cite{PV} we shall use the Radon transform
\begin{equation}
\label{eq3.1}
R(f)(s,\omega)=\D\int_{\{x\cdot \omega=s\}\cap\Omega}f d\mu_{s,\omega}
%no es si es \Omega o \Omega_j
\end{equation}
   
  with $\Omega$ either $\Omega_1$ or $\Omega_2$. In other words we extend $f$ as zero outside of $\Omega$ and compute its Radon transform.
  
  We set
  \begin{equation}
\label{eq3.2}
I_{\rho}(t)=I_{\rho}\big(u_1(t),u_2(t)\big)=\D\int_{\Omega_1\times\Omega_2}
\rho(x-y)\big|u_1(x)\big|^2\big|u_2(y)\big|^2 dx dy.
\end{equation}
  
  Then a simple modification of the proof of Theorem 2.5 in \cite{PV} gives us \footnote{There are some misprints in the second and fourth term of the RHS of (2.19) in \cite{PV}. A factor 2 is missing in both of them.}
  \begin{theorem}
  \label{Th1}
  Let $\omega\in\R^n$, $n>1$, with $|\omega|=1$ and $\rho_{\omega}(z)=|z\cdot \omega|$, $u_j$ solution of \eqref{eq2} with $j=1,2$. Then if $x=x^{\perp}+s\omega$ and $x^{\perp}\cdot\omega=0$ we have
  \begin{multline}
\label{eq4}
\frac{d^2}{dt^2}
I_{\rho_{\omega}}=\D\int_s\left|\partial_s\big(\mathcal R(u_1\overline
  u_2)\big)\right|^2ds\\
{}+
\epsilon\frac{p-1}{p+1}\left(\int_s\mathcal R\big(|u_1|^2\big)\mathcal R\big(|u_2|^{p+1}\big)ds+\D\int_s\mathcal R\big(|u_2|^2\big)\mathcal R\big(|u_1|^{p+1}\big)ds\right)\\
{}+\D\int_s\int_{x\cdot \omega=s}\int_{y\cdot\omega=s}
\left|u_1\left(x^{\perp}+s\omega\right)\partial_s
  u_2\left(y^{\perp}+s\omega\right)\right.\\
\left.{}- u_2\left(y^{\perp}+s\omega\right)\partial_s u_1\left(x^{\perp}+s\omega\right)\right|^2dx^{\perp}dy^{\perp}ds\\
{}-\D\int_{\partial\Omega_1\times\Omega_2}|u_2|^2(y)\partial_n\rho_{\omega}(x-y)\left|\partial_n
  u_1\right|^2(x)dS_xdy\\
{}-\D\int_{\Omega_1\times\partial\Omega_2}|u_1|^2(x)\partial_n\rho_{\omega}(x-y)\left|\partial_n u_2\right|^2(y)dxdS_y
\end{multline}
where $\partial_n$ is the outgoing normal vector field on $\partial\Omega$.
  \end{theorem}
  
  A simple look at the statement of the above theorem tells us that it
  is useless unless the boundary terms appearing in \eqref{eq4} are
  under control. In \cite{PV} this is done using the so--called local
  smoothing property that follows from a variation of Morawetz's
  multiplier. As a consequence a sufficient geometric condition on the boundary
  naturally appears: if $u$ is a solution to \eqref{eq2}, and $n(x)$
  denotes the outgoing normal vector to $\partial\Omega$ at $x$, local
  smoothing for the nonlinear solution $u$ is obtained provided
$$
\int_0^T\int_{\partial \Omega} (x\cdot n(x))|\partial_n u|^2 \,dS_x dt
\geq 0\,.
$$
In particular, it is sufficient to assume that $\Omega$ is the
  complement of a star--shaped domain and this is what we are going to
  assume in this paper. Unfortunately the use of Morawetz's multiplier
  introduces again the restriction on the dimension and therefore the
  result in \cite{PV} (Proposition 2.7) is given in dimension three
  and higher. In this paper we remove such a restriction.
  
  The main new idea is to use Theorem \ref{Th1} with
  $\Omega=\Omega_1=-\Omega_2$, $u_1=u$ and $u_2(x)=u(-x)$ and $u$ a
  solution of \eqref{eq1}. Equivalently we may consider
  $\widetilde{\rho}=\rho(x+y)$ instead of $\rho(x-y)$ in the
  definition of $I_{\rho}$ in \eqref{eq3.2}. We then sum up the new (
  with $\widetilde\rho_\omega$)
  and old (with $\rho_\omega$) estimates: it follows that some cancellation
  occurs between boundary terms of
  \eqref{eq4}, and integrating over $\omega\in\mathbb{S}^{n-1}$ one obtains that
  the boundary term is bounded above by
  $$
\D\int_{\partial\Omega\times\Omega}\big|\partial_n
  u(x)\big|^2\big|u(y)\big|^2{\langle y\rangle^{-1}}dS_x\,dy\,,$$
 where  the gain in the above inequality comes  from the weight
 $\langle y\rangle=\big(1+|y|^2\big)^{1/2}$.
  
  It turns out this gain is sufficient. The reason is that we can use
  the tensor product technique developed in  \cite{CGT} and
  \cite{CHVZ}. In order to do that  we construct $v(x,y)=u(x)u(y)$,
  solution of NLS in $\Omega\times\Omega$, and use the local smoothing
  inequality obtained from Morawetz's multipliers in dimension $n=4$, see Proposition \ref{L2} below. That fits perfectly well with our
  purposes.
\begin{remark}
  In the $\R^n$ case, all monotonicity formulae stemming from virial
  identities are related to the conservation of the momentum, a
  key property which is lost on domains. The boundary terms then appear as a direct consequence. The star-shaped case
  provides a class of obstacles for which this loss is controlable by
  an integration by parts argument (the usual virial in dimension
  $n\geq 3$ or the  tensorialized one in the present paper for
  $n=2$). For the linear equation, microlocal techniques allow to
  control the momentum through local smoothing estimates - as they
  control boundary terms, see \cite{PV} -, when the exterior domain is
  non trapping (light rays escape to spatial infinity). But
  extending such a control to the nonlinear setting appears a
  challenging task.
\end{remark}
 We denote by $H^1_0=H^1_0(\Omega )$ the energy space, which is also the domain of the square root of $-\Delta _\Omega $, where
$\Delta _\Omega $ is the Dirichlet Laplacian on $\Omega $. We are now able to state our main theorem.
    %Theorem 2
   \begin{theorem}
  \label{Th2}
  Let $\Omega$ be $\R^2\setminus \Sigma$, where $\Sigma$ is a star--shaped and bounded domain, and $u_0\in H_0^1(\Omega)$. Then, there exits a unique solution of
  $$\left\{\begin{array}{rcl}
  i\partial_t u+\Delta u&=&|u|^{p-1}u\qquad x\in\Omega\quad,\quad t\in\R\quad p\geq5,\\
  \left.u\right|_{\partial\Omega}&=&0,\\
  u(x,0)&=&u_0(x),
  \end{array}\right.$$
  such that
  $$u\in\mathcal C\big(\R:H_0^1(\Omega)\big)\cap L_t^{p-1}L_{x}^{\infty}.$$
  
  Moreover, there exist unique $u_+,u_-\in H_0^1(\Omega)$ such that
  \begin{equation}
\label{eq4'}
\left\|u(\cdot,t)-e^{it\Delta_{\Omega}}u_\pm\right\|=o(1)\quad t\to\pm\infty.
\end{equation}
Here $e^{it\Delta_{\Omega}}$ is the solution of the linear problem (i.e. \eqref{eq1} with $\epsilon=0$).

   \end{theorem}
 % \begin{remark}  \label{R4}
  %The proof of the above result easily extends to  nonlinearities $p>5$. We include some comments about the necessary modifications at the end of section 2.
  %\end{remark}
  
   The reader can guess that the key property of the solution obtained
   in the above theorem is that it satisfies the global in time
   Strichartz estimate $L_t^{p-1} L_{x}^{\infty}$. The key step in
   that direction stems from Theorem \ref{Th1}
   and the local smoothing estimate that we mentioned before: one gets
   that $D^{1/2}(|u|^2)$ is  in $L^2$ of both variables space and
   time. This is also true for linear solutions. This turns out to be
   the last key ingredient for the proof of Theorem \ref{Th2}. 
   We have the following result.
   
    %Theorem 3
   \begin{theorem}
  \label{Th3}
  Let $\Omega=\R^2\setminus \Sigma$, where $\Sigma$ is star--shaped and bounded. Define $e^{it\Delta_\Omega} u_0$ as the solution of
  $$\left\{\begin{array}{rcl}
  i\partial_t u+\Delta u&=&0\qquad x\in\Omega\quad,\quad t\in\R\\
  \left.u\right|_{\partial\Omega}&=&0\\
  u(x,0)&=&u_0(x),
  \end{array}\right.$$
  
  Then
$$
    \left\|D^{1/2}(|u|^2)\right\|_{L_t^2L_{x}^2} \le
    C\left\|u_0\right\|_{L^2}^{3/2}\left\|u_0\right\|_{H_0^1}^{1/2}\,,
$$
and, for $0\leq s\leq3/4$,
$$
\left\|D^s u\right\|_{L_t^4L_x^8}  \le C\left\|u_0\right\|_{H_0^{s+1/4}}\,.
$$
As a consequence
 \begin{equation}
   \label{eq:pinfty}
\left\|u\right\|_{L_t^4L_x^{\infty}}\le C\left\|u_0\right\|_{H_0^1}\,.
 \end{equation} 
 Here $D^s=(-\Delta)^{s/2}$ and ${H}_0^s(\Omega)={H}_0^s$ 
 denotes the domain of the operator $(-\Delta _\Omega
)^{s/2}$ given
by the spectral theorem applied to the Dirichlet Laplacian on $\Omega$.
\end{theorem}
 \begin{remark}
  \label{R4'}Notice that under the smoothness conditions on $\Omega$, $D^s\left(-\Delta_{\Omega}\right)^{-s/2}$ is bounded in $L^p\quad 1<p<\infty$. See for example \cite{IP} and references there in for the proof of this fact.
   \end{remark}
 \begin{remark}
  \label{R44} One may extend Theorem \ref{Th2} to nonlinearities with
  $p<5$, by using the known range of Strichartz estimates outside non
  trapping obstacles from \cite{BSS}: for $\varepsilon>0$,
$$
\| u\|_{L^{\frac 3 {1-\varepsilon}}_t L^\frac 2 \varepsilon_x} \leq
C(\varepsilon) \|u_0\|_{H^{\frac 1 3(1-\varepsilon)}_0}\,.
$$
This allows to replace $p=4$ in Theorem \ref{Th3} by $p=3/(1-\varepsilon)$,
which will however not fill the
  expected range $3<p<5$ in Theorem \ref{Th2}, by analogy with the $\R^2$ case, but rather
  provide $p>4$. We elected to keep the argument
  mostly self-contained, as the linear estimates from Theorem
  \ref{Th3} are enough for our purposes in the range $p\geq 5$.
   \end{remark}
   In the next section  we prove Theorems \ref{Th2} and \ref{Th3} and give the necessary propositions to be able to use Theorem \ref{Th1}. This one is proved as Theorem 2.5 in \cite{PV} and therefore we omit the details.   %%In Section 3 we state the analogous results for non--linearity $|u|^{p-1}u$ with $p>5$ and give some hints of the necessary modifications for the proofs.
  
  \section{Proofs.}
  \label{S2}
  We will first prove Theorem \ref{Th3}. We need some lemmas.
  
   %Lemma 1
   \begin{lemma}
  \label{L1}
  For $\omega\in\mathcal S^1$ define $\rho_{\omega}(x)=|x\cdot\omega|$ and 
  $$A_{\omega}=\left\{(x,y)\in\partial\Omega\times\Omega\,:\,|y\cdot\omega|\le|x\cdot\omega|\right\}.$$
  Then if $y\in\Omega$
    \begin{equation}
\label{eq6}
\renewcommand{\arraystretch}{1.5}
\begin{array}{l}
\left\{(x,y)\in\partial\Omega\times\Omega\,:\,\partial_n\rho_{\omega}(x-y)+\partial_n\rho_{\omega}(x+y)\neq0\right\}\\
\subset\left\{\text{sig}\,(x-y)\cdot\omega+\text{sig}\,(x+y)\cdot\omega\neq 0\right\}=A_{\omega},
\end{array}
\renewcommand{\arraystretch}{1}
\end{equation}
and
  \begin{equation}
\label{eq5}
\sup_{x\in\partial\Omega}\left|\D\int_{\mathcal S^1}\chi_{A_{\omega}}d\sigma(\omega)\right|
\le\frac{C}{\langle y\rangle},
\end{equation}

   \end{lemma}

   %proof lemma 1
   \begin{proof}
        We can assume without loss of generality that $\omega=(1,0)$. Then if $x=(x_1,x_2)$, $y=(y_1,y_2)$ we have that
   $$
  \renewcommand{\arraystretch}{1.5} \begin{array}{rcl}
   \partial_n\rho_{\omega}(x-y)+ \partial_n\rho_{\omega}(x+y)&=&\text{sig}\,(x-y)\cdot\omega+\text{sig}\,(x+y)\cdot\omega\\
   &=&\text{sig}\,\left(x_1-y_1)\right)+\text{sig}\,\left(x_1+y_1)\right)
   \end{array}\renewcommand{\arraystretch}{1},
   $$
   and hence by inspection we get \eqref{eq6}
   
   As for \eqref{eq5} we have that if $|y\cdot\omega|<|x\cdot\omega|$
   then
   $|y|\left|\cos(y,\omega)\right|\le\text{diam}\,\left(\partial\Omega\right)$. Thus,
   for $|y|$ large enough
   $$
|y|\left|\theta_y-\theta_{\omega}\pm\D\frac{\pi}2\right|\le M\,
$$
   with $\frac y{|y|}=\left(\cos \theta_y,\sin
     \theta_y\right)$ and $\omega=\left(\cos \theta_{\omega},\sin \theta_{\omega}\right)$.
   The lemma easily follows.
   \end{proof}   
     %Lemma 2
   \begin{proposition}
  \label{L2}
  Let $\Omega$ be $\R^2\setminus\Sigma$, with $\Sigma$ a bounded smooth star--shaped obstacle. Assume $u$ is a solution of
   $$\begin{array}{rcl}
   i\partial_t u+\Delta u&=&\epsilon |u|^{p-1}u,\quad p>1,\\
   \left.u\right|_{\partial\Omega}&=&0,
   \end{array}$$
   with $\epsilon\in\{0,1\}$. Then
  \begin{equation}
\label{eq7}
\D\int_{\R}\int_{\partial\Omega\times
  \Omega}\left|\partial_nu(x)\right|^2|u(y)|^2 dS_x\frac{dy}{\langle
  y\rangle}dt\le CM^{3/2} E^{1/2}
\end{equation}
with $M$ and $E$ as in \eqref{eq3}.
     \end{proposition}
     \begin{proof}
       Define $v(x,y)=u(x)u(y)$ solution of the problem
   \begin{equation}
     \label{eq:tnls}
     \renewcommand{\arraystretch}{1.5}\begin{array}{rcl}
   i\partial_tv+\Delta v&=&\epsilon (|u|^{p-1}(x)+u^{p-1}(y))v,\qquad (x,y)\in\Omega\times\Omega,\\
   \left.v\right|_{\partial(\Omega\times\Omega)}&=&0,\\
   v(x,y,0)&=&u_0(x)u_0(y).
   \end{array} \renewcommand{\arraystretch}{1}
   \end{equation}
   
   Then, consider for $h(x,y)=\sqrt{|x|^2+|y|^2}$
   $$M_h(t)=\D\int_{\Omega\times\Omega}|v|^2(x,y,t)h(x,y)dx dy$$
   and compute $\D\frac{d^2}{dt^2}M_h(t)$. The details can be found
   for example in \cite{PV} (p. 278), up to easy modifications to
   handle the nonlinear term in \eqref{eq:tnls}. We just have to bother about $\partial_n h$ with $n$ the normal to $\partial(\Omega\times\Omega)$, namely
   $$\renewcommand{\arraystretch}{1.5}\begin{array}{rcl}
   n=\left(n_x,0\right)&\text{if}&x\in\partial\Omega\quad,\quad y\in\Omega,\\
   n=\left(0,n_y\right)&\text{if}&x\in\partial\Omega\quad,\quad y\in\Omega.
   \end{array}\renewcommand{\arraystretch}{1}$$
   
   Hence $$\partial_n h(x,y)=\D\frac{n_x\cdot x}{\sqrt{|x|^2+|y|^2}}$$
   if $(x,y)\in\partial\Omega\times\Omega$ and similarly in the other
   case. We recall that the kernel of a star-shaped domain is the set
   of points with respect to which the domain is star-shaped. If we
   assume that the kernel of our obstacle contains a disk (note that
   the kernel is always convex) then we may exclude the situation
   where $n_x\cdot x=0$, by averaging over this disk if
   necessary.  Abusing notation by forgetting about this possible
   average over the base point, we then get that
   $$\left|\partial_n h(x,y)\right|\ge\D\frac{C}{|y|},$$
   with $C>0$. The argument follows as in pg 278 of \cite{PV} because
   $$\left|\partial_n v\right|^2=\left|\partial_n u(x)\right|^2\left|u(y)\right|^2$$ 
   if $(x,y)\in\partial\Omega\times\Omega$ and a similar expression if $(x,y)\in\Omega\times\partial\Omega$.
\end{proof}

If the kernel of the obstacle is a single point (or even a segment), we
need to address the situation where there exist points
$x\in \partial\Omega$ such that $n_x\cdot x=0$. We may repeat the
computation of \cite{PV} with weight $\rho(x,y)=\sqrt{1+|x|^2+|y|^2}$, but
the desired control on the boundary does not follow by a subsequent
computation with a weight like the distance to the boundary, because
in our setting the 4D obstacle is not compact and we need to account
for the trace of $\partial_n h$ on the boundary.  Let us
summarize what the computation with weight $\rho$ provides, following \cite{PV}:
\begin{lemma}\label{Ls4d}
  In the conditions of Proposition \ref{L2}, 
  \begin{equation}
    \label{eq:4Dsmoothing}
    \int_{\R} \int_{\Omega\times\Omega} \frac{|\angnabla v|^2}{\rho(x,y)}+
    \frac{|\nabla v|^2+|v|^2}{\rho^3(x,y)}+
    \frac{(|u|^{p-1}(x)+|u|^{p-1}(y)) |v|^2}{\rho(x,y)}
    \,dxdy \le CM^{3/2}E^{1/2}\,,
  \end{equation}
where $\angnabla$ denotes the (4D) angular gradient.
\end{lemma}
We now proceed with the remaining case in the proof of Proposition \ref{L2}~: control
of the boundary term in the situation where there exist points
$x\in\partial \Omega$ such that $x\cdot n(x)=0$. While Lemma
\ref{Ls4d} is not enough to conclude directly, it will be a key
ingredient in what follows. Recall we aim at controling the lefthand
side of \eqref{eq7}.

 Let $R$ be such that $\partial\Omega\subset B(0,R-1)$. For the part of
the 4D boundary $(\partial\Omega\cap B(0,R))^2$ we may proceed as in
\cite{PV}, using the control of $\nabla v$ in \eqref{eq:4Dsmoothing},
because the weight $\rho^{3}$ is irrelevant. The remaining part of the 4D
boundary is a union of cylinders $\partial\Omega \times (\R^2\setminus
B(0,R))$ and $(\R^2\setminus B(0,R))\times \partial\Omega$. They will
be treated similarly and we restrict our attention to the first one
(boundary in $x$ and exterior of a large ball in $y$). We may further
divide the $y$-region in a finite number of angular sectors: without
loss of generality we now assume to be in such a sector $S=\{|y_2|<|y_1|/10\}$. Notice how being outside a ball forces $|y_1|>R/2$.

 From the
compactness of $\partial\Omega$ and its smoothness as a curve in
$\R^2$, there exists a vector field $Z$ such that $Z$ is defined in
$\Omega$, its restriction on $\partial\Omega$ is the normal
derivative, and it vanishes outside $B(0,R)$. In a coordinate system,
$$
Z=a(x)\partial_{x_1}+b(x)\partial_{x_2},
$$
where $a(x)$, $b(x)$ are smooth functions vanishing for $|x|>R$. We define
two  4D vector fields on our 4D region of interest, $(\Omega\setminus
B(0,R))\times S$:
$$
Y_k= y_1 \partial_{x_k}-x_k\partial_{y_1}\,, \text{ with
} k=1,2.
$$
Notice that $Y_1$ and $Y_2$ are, up to a factor $({\sqrt{|x|^2+|y|^2}})^{-1}$, angular derivatives (in 4D). We may
now define
$$
Y= \alpha(x,y) Y_1+\beta(x,y) Y_2
$$
where the support of $\alpha$ and $\beta$ is in $(\Omega\setminus
B(0,R))\times S$ and their respective value is adjusted such that
$(Y v)_{|x\in \partial\Omega} \approx h^{-1}(x) Z v$, as $\partial_{y_1}
v=0$ on $\partial\Omega \times \Omega$: set
$\alpha(x,y)=\phi(y/|y|) a(x)/y_1^2$ and $\beta(x,y)=\phi(y/|y|)b(x)/y_1^2$ where
$\phi(y/|y|)$ is the identity in the angular sector $S$, and vanishes
outside  $2S$. One should think of the following computation as an
analog of the one in \cite{PV} p. 278/279, but where we think of the
region of 4D space we are integrating over as being foliated by
spheres $|x|^2+|y|^2=C$ rather than planes $|y|=C$. We use the following momentum
$$
J=2 \text{Im} \int_{\Omega\times \Omega} \bar v (Y v)\, dxdy.
$$
For the sake of notational convenience, we set $X=(x,y)$ and
$Y=A(X)\cdot \nabla_X=A\cdot \nabla$ and $N(u)=|u|^{p-1}(x)+|u|^{p-1}(y)$. The following computation is standard, does
not depend on the specific form of $A$ and similar to the one in
\cite{PV} where $A=\nabla h$;
\begin{align*}
  \partial_t \left(2 \text{Im} \int \bar v \nabla v \cdot A \right) = &
   - 2 \text{Im} \int \partial_t v ( 2 \nabla \bar v \cdot A +\bar
  v  \nabla\cdot A) \\
 = & \,  -4 \text{Re} \int \Delta v \nabla \bar v \cdot A + 2 \int |\nabla
 v|^2 \nabla\cdot A-\int |v|^2 \Delta \nabla\cdot A \\
 & {}+2 \int N \nabla (|v|^2) \cdot A+|v|^2 N \nabla \cdot A\,.
\end{align*}
We have then (with $A_n$ the normal component of $A$ on the boundary
and $d\sigma$ the surface measure on it, with orientation toward the domain)
\begin{align*}
2 \text{Re} \int \Delta v \nabla \bar v \cdot A = & -2 \int |\partial_n
v|^2 A_n d\sigma-2\text{Re}\int \nabla v \cdot \nabla (A\cdot \nabla
\bar v)\\
= & - 2 \int |\partial_n
v|^2 A_n d\sigma-2\text{Re}\int A\cdot \nabla (\nabla \bar v\cdot
\nabla v) \\
 & {}+2\text{Re}\int [A\cdot \nabla,\nabla] \bar v\cdot
\nabla v
\end{align*}
and by substitution and integration by parts in the middle term and
the nonlinear term above,
\begin{multline*}
  \partial_t \left(2 \text{Im} \int \bar v \nabla v \cdot A \right) =  2 \int |\partial_n
v|^2 A_n d\sigma-4 \text{Re}\int [A\cdot \nabla,\nabla] \bar v\cdot
\nabla v\\
 {}-\int |v|^2 \Delta \nabla\cdot A +2 \frac{p-1}{p+1}\int  |v|^2
 (|u|^{p-1}(x) \nabla_x \cdot A_x+|u|^{p-1}(y)\nabla_y\cdot A_y)\,,
\end{multline*}
where $A_x$ (resp.  $A_y$) stands for the projection over the $x$
coordinates (resp. the $y$ coordinates) of the vector field $A$. Given our choice of $Y=A\cdot \nabla$, the boundary term is what we
seek, up to an harmless factor $h(x,y)/y_1$:
$$
\int |\partial_n
v|^2 A_n d\sigma=\int Yv \nabla \bar v \,d\sigma=\int_{\partial\Omega\times \Omega} \frac{\phi(y/|y|)}{y_1} |\partial_n
v|^2 \,dS_x dy\,.
$$
The term $N(u)|v|^2$ will be under control by Lemma
\ref{Ls4d}. For the $|v|^2$ term, we may use Poincar\'e's inequality
on the spheres $|x|^2+|y|^2=C$ to control it as well by $\int
\rho^{-1}(x,y) |\angnabla v|^2$ (notice how when deriving in $x$ we do
not gain decay in the coefficients of $Y$).

We are left with the term carrying $[Y,\nabla]$: for any scalar function $\gamma$, we have
$[\gamma Y_k,\partial]=\gamma [Y_k,\partial]-(\partial \gamma)
Y_k=-(\partial \gamma) Y_k$ ; hence the vector field $[Y,\nabla]$ is spanned by
angular derivatives.  Moreover, $|\partial \gamma|\lesssim
|\gamma|$ (as the worst situation is when the derivative would hit
$a(x)$ or $b(x)$, which does not gain decay in $y$) and $\gamma \approx y_1/\rho(x,y)$, so that
$$
\int | [Y,\nabla] \bar v\cdot\nabla v| \lesssim \int \frac{|\angnabla v|^2}{\rho(x,y)}
$$
which is controlled by Lemma \ref{Ls4d} after time
integration. This concludes the proof of Proposition \ref{L2}.
\qed
    %proposition 1
   \begin{proposition}
  \label{P1}
  In the conditions of Proposition \ref{L2}
  \begin{equation}
\label{eq8}
%dobe barra o simple
\D\int_{\R}\int_{\Omega\times\Omega}\left|D^{1/2}|u|^2\right|^2 dx dy\le CM^{3/2}E^{1/2}.
\end{equation}
   \end{proposition}
   
   %proof
   \begin{proof}
 We use Theorem \ref{Th1} with $u_1=u(x)$ and $u_2=u(-x)$, $\Omega'=-\Omega$. Then if $y\in\Omega'$
  $$\partial_n\rho_{\omega}(x-y)=\partial_n\rho_{\omega}(x+y).$$
  
  We use again Theorem \ref{Th1} with $u=u_1=u_2$ and
  $\Omega_1=\Omega_2=\Omega$. Then we sum the two
  identities and integrate in
  time once. After discarding some positive terms we get

  \begin{multline*}
  \D\int_0^T\int_s\left|\partial_s\mathcal R|u|^2(s,\omega)\right|^2 ds d\omega\\
  {}-2\D\int_{\partial\Omega\times\Omega}\left|\partial_n u(x)\right|^2|u(y)|^2\left(\partial_n\rho_{\omega}(x-y)
 + \partial_n\rho_{\omega}(x+y)\right)dS_x dy\\
{}-2\D\int_{\Omega\times\partial\Omega}\left|\partial_n u(y)\right|^2|u(x)|^2\left(\partial_n\rho_{\omega}(x-y)
 + \partial_n\rho_{\omega}(x+y)\right)dxdS_y\\
     \le C M^{3/2}E^{1/2},
  \end{multline*}
  where the right hand side follows from the trivial bounds of $I'_{\rho_{\omega}}$.
  
  We conclude the argument integrating on $\omega\in\mathbb{S}^1$ and using the well known Plancherel's Theorem for the Radon transform as in \cite{PV}, (Prop 2.2) to obtain 
  $$
\left\|D^{1/2}(|u|^2)\right\|_{L^2_{x,[0,T]}}.
$$
  For  the boundary term we use Lemma \ref{L1} and then Proposition
  \ref{L2}. Notice that both lemmas are symmetric in the $(x,y)$
  variables.
   \end{proof}
  \begin{remark}
The astute reader will have noticed by now that the entire computation
may be performed directly, with the (integrated over $\omega$)
weight $\rho(x,y)=|x-y|+|x+y|$. We feel  the directional weight to be
more natural, with the key Lemma \ref{L1} having a simple geometrical
proof. In fact, we were led to the present result from considering
simple obstacles with symmetries (e.g. the disk).
  \end{remark}
    %proposition 2
   \begin{proposition}
  \label{P2}
  In the conditions of Lemma \ref{L2} and $\epsilon=0$,
  \begin{equation}
\label{eq9}
%dobe barra o simple
\left\|D^s  u\right\|_{L^4_tL^8_x}\le C\|u_0\|_{{H}_0^{s+1/4}}
\end{equation}
   \end{proposition}
   %proof
   \begin{proof}
Take $u_0$ a frequency localized function (according to the Dirichlet
Laplacian of $\Omega$). Then \eqref{eq9} is a consequence of
\eqref{eq8} and Sobolev embedding. The general case follows by the
square function estimates on domains and the boundedness in $L^p$,
$1<p<\infty$, of the corresponding Riesz transforms. see for example
\cite{IP}, Theorem 1.2 and Remark (1.10).

\end{proof}
Notice that Proposition \ref{P1} and Proposition \ref{P2} imply
Theorem \ref{Th3}. We already have all the necessary ingredients for
the proof of Theorem \ref{Th2}.
\begin{proof}[Proof of Theorem \ref{Th2}]
For simplicity we will consider the particular case $p=5$. The necessary changes needed for the rest of the powers will be outlined at the end of the proof.
  
   We first notice that by Gagliardo--Nirenberg inequality
  \begin{equation}
\label{eq10}
\|f\|_{L^{\infty}(\R^2)}\le C\|f\|_{L^8(\R^2)}^{2/3}\left\|D^{3/4}f\right\|_{L^8(\R^2)}^{1/3}.
\end{equation}

Hence
\begin{equation}
\label{eq11}
\left\|e^{it\Delta_{\Omega}}u_0\right\|_{L_t^4L_x^{\infty}}\le C\|u_0\|_{H_0^{s}},\qquad s>1/2.
\end{equation}
This inequality together with the trivial one
$$\left\|\nabla|u|^4u\right\|_{L^2}\le c\|u\|_{L^{\infty}}^4\|\nabla u\|_{L^2},$$
allows us to do a fixed point argument in $L_t^4L_x^{\infty}\cap L_t^{\infty}H_0^1$, to solve locally in time
$$u(t)=e^{it\Delta_{\Omega}}u_0+i
\D\int_0^te^{i(t-\tau)\Delta_{\Omega}}|u|^4u(\tau) d\tau.$$

On the other hand given any interval $I$ of time where the solution exists we know from Proposition \ref{P1} that $\|u\|_{L_I^4 L_x^8}$ is finite and bounded by $E^{1/8}M^{3/8}$. Hence given any $\epsilon>0$ there is a finite number of disjoint intervals $I_1,\dots, I_N$ such that $\bigcup\limits_{j=1}^N I_j=I$ with $N=N(\epsilon)$ and
$$\|u\|_{L_{I_j}^4 L_x^{\infty}}^4= \epsilon,\quad j<N(\epsilon), \quad
\|u\|_{L_{I_j}^4 L_x^{\infty}}^4\leq \epsilon,\quad j=N(\epsilon).$$
 Hence 
$$N(\epsilon\leq CE^{1/2}M^{3/2}\epsilon^{-1}.$$

Assume $I_j=[t_j,t_{j+1})$ so that 
$$\|u(t_j)\|_{H_0^1}\le E^{1/2},$$
and call
$$\sigma(t)=\D\int_{t_j}^{t}\|u\|_{L_x^{\infty}}^4 dt.$$
Then $\sigma(t)$ is a continuous function with $\sigma(t_j)=0$.
Moreover for $t_j\le t<t_{j+1}$
$$\left\|D^{3/4}u\right\|_{L_{[t_j,t]}^4 L_x^8}\le
CE^{1/2}\left(1+\|u\|^4_{L_{[t_j,t]}^4L_x^\infty}\right).$$

Hence
$$\renewcommand{\arraystretch}{1.8}\begin{array}{rcl}
\sigma(t)&\le&\|u\|_{L_{[t_j,t]}^4 L_x^8}^{8/3}
\left(CE^{1/2}\left(1+\sigma(t)\right)\right)^{4/3}\\
&\le&C\epsilon^{2/3}E^{2/3}(1+\sigma(t)^{4/3}),
\end{array}\renewcommand{\arraystretch}{1}$$
and taking $\epsilon E$ small enough, (i.e. $ \epsilon=(CE)^{-1}$) we conclude that $\sigma(t_{j+1})$ remains bounded by a universal constant independent of $I$. Therefore the solution is global and
$$\|u\|^4_{L_{\R}^4L_{x}^{\infty}}\le CN\le CE^{3/2}M^{3/2}.$$

Finally defining
$$u_+=u_0+\D\int_0^{\infty}e^{i\tau\Delta_{\Omega}}|u|^4u(\tau) d\tau$$
we obtain \eqref{eq4'}.

For the general case $p>5$ we need a bound for $\|u\|_{L_{t}^{p-1} L_x^{\infty}}.$ Call $q=p-1$. Then for $\theta=4/q$ we have 
$$\|u\|^q_{L_{t}^{q} L_x^{2q}}\leq C\int_0^{+\infty} \left(\|u\|^\theta_{L_x^{8}}
\|u\|^{1-\theta}_{H^1_0}\right)^q\, dt\leq CM^{3/2}E^{\frac{q}{2}-\frac{3}{2}}.$$
Hence
$$\|e^{it\Delta_\Omega} u_0\|^q_{L_{t}^{q} L_x^{2q}}\leq CM^{3/2q}E^{\frac{1}{2}(1-\frac{3}{q}}).$$
Finally it is enough to use instead of \eqref{eq10} the following inequality
$$\|f\|_{L^{\infty}(\R^2)}\le C\|f\|_{L^{2q}(\R^2)}^{2/3}\left\|D^{3/q}f\right\|_{L^{2q}(\R^2)}^{1/3}, \quad q\geq 3.$$
The rest of the proof works similarly.
\end{proof}
\begin{remark}
Notice how, in the previous argument, one is combining a  Strichartz estimate of type $L^{r}_t
L^\infty_x$ with the $H^1_0$ norm to control
the source term $|u|^{p-1} u$ in $L^1_t(H^1_0)$, and this in turn
yields $r=p-1$. Given that all available Strichartz estimates on
domains are suboptimal with respect to regularity - there is a loss in
the Sobolev embedding scale when compared to $\R^2$ - we cannot hope
to use e.g. dual Strichartz pairs to allow for higher time
integrability. As such, dealing with a given $p$ requires the
availability of Strichartz estimates with $r=p-1$. In light of
ongoing developpement on linear dispersion on domains, we may hope
(being optimistic) for $r>12/5$, which is still not the $\R^2$ range,
and new insights on the linear theory will be needed to match it.
\end{remark}
    \bibliographystyle{plain}
    \bibliography{NLS-PV}

\end{document}